 \documentclass[final,3p,times,sort&compress]{elsarticle}

\usepackage{mathrsfs}        
\usepackage{graphicx}        
\usepackage{amssymb}         
\usepackage{amsmath}         
\usepackage{enumerate}       
\usepackage[colorlinks,
            linkcolor=red,
            anchorcolor=blue,
            citecolor=green
            ]{hyperref}
\usepackage{caption}
\usepackage{amsthm}          

\newtheorem{theorem}{Theorem}
\newtheorem{remark}{Remark}
\newtheorem{example}{Example}
\newtheorem{lemma}{Lemma}

\DeclareCaptionLabelSeparator{dotmark}{. }
\begin{document}

\begin{frontmatter}

\phantomsection
\addcontentsline{toc}{chapter}{Sufficient and necessary conditions for stabilizing singular fractional order systems with partially measurable state}
\title{Sufficient and necessary conditions for stabilizing singular fractional order systems with partially measurable state}


\author{Yiheng Wei, Jiachang Wang, Tianyu Liu, Yong Wang$^*$}
\address{Department of Automation, University of Science and Technology of China, Hefei, 230026, China}

\cortext[cor1]{Corresponding author. E-mail: yongwang@ustc.edu.cn.}

\begin{abstract}
This paper is concerned with the stabilization problem of singular fractional order systems with order $\alpha\in(0,2)$. In addition to the sufficient and necessary condition for observer based control, a sufficient and necessary condition for output feedback control is proposed by adopting matrix variable decoupling technique. The developed results are more general and efficient than the existing works, especially for the output feedback case. Finally, two illustrative examples are given to verify the effectiveness and potential of the proposed approaches.
\end{abstract}


\end{frontmatter}


\section{Introduction}\label{Section1}
Fractional order systems (FOSs) represent a useful class of systems which are more complicated than classical integer order systems. They often appear in various practical applications such as viscoelastic systems, electrochemistry, economy and biology systems \cite{Mathiyalagan:2015Com,Zou:2017JMCI,Sun:2018CNSNS}. Due to its importance in both theoretical study and practical applications, such systems attract increasing attention, especially with respect to system identification \cite{Cui:2017NODY,Cheng:2018SP}, stability analysis \cite{Wei:2017FCAA,Huang:2016CTA}, controller synthesis \cite{Sheng:2017JFI,Chen:2018IJRNC} and numerical computing \cite{Wei:2016ISA,Stanislawski:2017JFI}, etc.

As a particular-yet wide-class of systems, singular systems are more general than standard ones since they can be described by a mixture of differential equations and algebraic equations. Accordingly, singular systems have received much attention in the past few decades \cite{Duan:2010Book}. Many valuable results on singular FOSs have been reported too \cite{Yin:2015AMC,Liu:2016NODY,Lazarevic:2016JVC,Jmal:2018NODY}. Based on the Gronwall's approach, the sufficient condition for the finite-time stability criterion was developed for stochastic singular FOS with $0<\alpha<1$ \cite{Mathiyalagan:2016Com}. With the idea of regularization, \cite{NDoye:2013Auto,Wei:2017ISA} tried to transform singular FOSs into normal ones and then to control them. In contrast, the admissibility analysis and controller synthesis on singular FOSs with order lying in $(0,2)$ were comprehensively studied in a pioneering thesis \cite{Xu:2009Master}. On the basis of this work, some scholars solved some other related problems \cite{NDoye:2010MCCA,Song:2012ICAL}. To remove the equality constraint in admissibility criterion, Yu proposed a criterion in the form of pure linear matrix inequality (LMI) for the $0<\alpha<1$ case \cite{Yu:2013AAS}. With this LMI criterion without equality constraints, the corresponding control problem was done subsequently \cite{Ji:2015ISA}. To reduce the number of decision matrix variables, Marir provided a more concise criterion \cite{Marir:2017IJCAS}. However, it involves complex decision matrix variables. To overcome this deficiency, an elegant and efficient LMI criterion was derived by introducing the concept of fractional order positive definite matrix \cite{Wei:2016CCC}. Similarly, an LMI criterion was presented from the region stability theory and the nonsingular decomposing technique \cite{Zhang:2017ISA}. Besides, Marir developed an admissible criterion for the $1<\alpha<2$ case \cite{Marir:2017JFI} which has the similar form with that in \cite{Yu:2013AAS}. Since it is not practical or even impossible to access all the state, the partial state control problem of singular FOSs becomes meaningful. \cite{Marir:2017JFI} studied the output feedback control problem for singular FOSs, while the LMI conditions in Theorem 3 are wrongly derived. The given criterion in \cite{Marir:2017JFI} was used in observer based control problem \cite{Lin:2018SCL}. \cite{Marir:2017IJCAS} also studied the observer based control problem for the $0<\alpha<1$ case, while Theorem 3 was also questionable. Motivated by the previous discussions, some observations can be reached. First, a criterion without inequality constraint, complex operations and many decision matrices is expected for observer based control. Nonetheless, the coupling phenomenon usually appears in such problem. Second, the sufficient and necessary condition for output feedback control is expected, since, to the best knowledge of the authors, the necessity is still an open problem for $\alpha=1$ and $E=I$. Third, a unified solution for both $0<\alpha<1$ and $1<\alpha<2$ cases is desirable, which will bring many theoretic and practical difficulties. In view of the above observations, this paper focuses on these challenging and essential problems.

The remainder of the paper is arranged as follows. Some preliminary results are provided along with the problem formulation in Section \ref{Section2}. In Section \ref{Section3}, the stabilization conditions of singular FOSs with two kinds of controllers are presented with rigorous proof. In Section \ref{Section4}, the effectiveness of our results is evaluated by two numerical examples. Finally, conclusions of this study are provided in Section \ref{Section5}.

\textbf{Notations.} $M^{\rm T}$ is the transpose of matrix $M$, $M^{\rm - 1}$ is the inverse of matrix $M$, ${\rm sym}(M)$ is used to denote $M^{\rm T}+M$, ${\rm det}(M)$ is the determinant of matrix $M$, ${\rm rank}(M)$ denotes the rank of matrix $M$, $I_n$ represents $n$ dimensional identity matrix, $*$ indicates the symmetric part of a matrix, such as $\Big[ {\begin{smallmatrix}
A&B\\
*&C
\end{smallmatrix}} \Big] = \Big[ {\begin{smallmatrix}
A&B\\
{{B^{\rm{T}}}}&C
\end{smallmatrix}} \Big]$ where $A = {A^{\rm{T}}}$ and $C = {C^{\rm{T}}}$, ${\deg}(f(s))$ means the degree of polynomial $f(s)$ and ${\rm arg}(z)$ stands for the angle of a complex number $z$.

\section{Preliminaries}\label{Section2}

In this work, we consider the following singular FOS
\begin{equation}\label{Eq1}
\left\{\begin{array}{rl}
E{{\mathscr D}^\alpha }x\left( t \right) =&\hspace{-6pt} Ax\left( t \right)+Bu\left( t \right),\\
y\left( t \right) =&\hspace{-6pt} Cx\left( t \right),
\end{array}\right.
\end{equation}
where $x\left( t \right)\in\mathbb{R}^n$ is the pseudo semi-state, $u\left( t \right)\in\mathbb{R}^m$ is the control input, $y\left( t \right)\in\mathbb{R}^p$ is the measured output, which is just the mentioned partial measurable state, $0<\alpha<2$ is the system commensurate order, $E\in\mathbb{R}^{n\times n}$ is a singular matrix such that ${\rm rank}(E)=r<n$ and $A,B$ and $C$ are constant matrices with appropriate dimensions.

The $\alpha$-th Caputo fractional derivative of $f(t)$ is defined by \cite{Monje:2010Book}
\begin{equation}\label{Eq2}
\begin{array}{l}
{\mathscr D}^\alpha f\left( t \right) \triangleq \frac{1}{{\Gamma \left( {n - \alpha } \right)}}\int_{0}^t {{{\left( {t - \tau } \right)}^{n - \alpha  - 1}}f^{(n)}\left( \tau  \right){\rm{d}}\tau },
\end{array}
\end{equation}
where $n-1<\alpha<n$, $n\in\mathbb{N}_+$ and $\Gamma \left( \cdot  \right) $ is the {\rm Gamma} function. Due to its good properties on Laplace transform and constant's derivative, the Caputo definition is adopted throughout this paper \cite{Monje:2010Book}.

Before moving on, the basic definitions \cite{Xu:2009Master} and relevant facts for singular FOSs are needed.

\begin{enumerate}[i)]
  \item The pair $\{E,A\}$ is said to be regular if ${\det \left( {sE - A} \right)}\not\equiv 0 $ for $s\in\mathbb{C}$.\vspace{-3pt}
  \item The pair $\{E,A\}$ is said to be impulse free if $\deg \left( {\det \left( {sE - A} \right)} \right) = {\rm{rank}}\left( E \right)$.\vspace{-3pt}
  \item The pair $\{E,A\}$ is said to be stable if all the finite eigenvalues of $\det \left( {\lambda E - A} \right) = 0$ lie in ${\mathcal D}_\alpha \triangleq\left\{ {\lambda :\left| {\arg \left( \lambda  \right)} \right|> \alpha \pi /2} \right\}$.\vspace{-3pt}
\end{enumerate}

The system (\ref{Eq1}) with $u(t)=0$ is said to be admissible if $\{E,A\}$ is regular, impulse free and stable. The regularity guarantees the existence and uniqueness of a time-domain response to a given singular FOS, while non-impulsiveness ensures no infinite dynamical modes in such system. Stability corresponds to a convergent system output.

The $n$-dimensional fractional order positive definite matrix set is defined by ${\mathbb{P}^{n \times n}_\alpha} \triangleq  \big\{ \sin \big( {\frac{{\alpha \pi }}{2}} \big)X+\cos \big( {\frac{{\alpha \pi }}{2}} \big) Y:X,Y \in {\mathbb{R}^{n \times n}},\left[ {\begin{smallmatrix}
X&Y\\
{ - Y}&X
\end{smallmatrix}} \right] >0 \big\}$ with $\alpha\in(0,1)$. With the well defined matrix set, the following lemmas provide two dual sufficient and necessary criteria for the admissibility of system (\ref{Eq1}) with $u(t)=0$ and $0<\alpha<1$.

\begin{lemma}\label{Lemma1} \cite{Wei:2016CCC}
The system (\ref{Eq1}) is admissible if and only if there exist matrices $P \in \mathbb{P}_\alpha ^{n \times n}$ and $Q \in {{\mathbb{R}}^{\left( {n - r} \right) \times n}}$, such that
\begin{equation}\label{Eq3}
{\rm{sym}}( {AP{E^{\rm{T}}} + A{E_0}Q} ) < 0,
\end{equation}
where ${E_0} \in {{\mathbb{R}}^{n \times \left( {n - r} \right)}}$, ${\rm rank}\left(E_0\right)=n-r$ and $E{E_0} = 0$.
\end{lemma}
\begin{lemma}\label{Lemma2} \cite{Wei:2016CCC}
The system (\ref{Eq1}) is admissible if and only if there exist matrices $P \in \mathbb{P}_\alpha ^{n \times n}$ and $Q \in {{\mathbb{R}}^{n\times\left( {n - r} \right) }}$, such that
\begin{equation}\label{Eq4}
{\rm{sym}}( {{E^{\rm{T}}PA} + Q{E_0}A} ) < 0,
\end{equation}
where ${E_0} \in {{\mathbb{R}}^{ \left( {n - r} \right)\times n}}$, ${\rm rank}\left(E_0\right)=n-r$ and ${E_0}E= 0$.
\end{lemma}

\section{Controller Design}\label{Section3}
\subsection{Observer based control}
Design the following observer based controller
\begin{eqnarray}\label{Eq5}
\left\{ \begin{array}{rl}
E{{\mathscr D}^\alpha }\hat x\left( t \right) =&\hspace{-6pt} A\hat x\left( t \right) + Bu\left( t \right) + L\left[ {C\hat x\left( t \right) - y\left( t \right)} \right],\\
u\left( t \right) =&\hspace{-6pt} K\hat x\left( t \right),
\end{array} \right.\hspace{-6pt}
\end{eqnarray}
where $\hat x\left( t \right)\in\mathbb{R}^n$ is the estimation pseudo state, $L\in\mathbb{R}^{n\times p}$ and $K\in\mathbb{R}^{m\times n}$ are, respectively, observer gain and controller gain to be designed, such that the augmented system
\begin{equation}\label{Eq6}
\bar E{{\mathscr D}^\alpha }\bar x\left( t \right) = \bar A\bar x\left( t \right),
\end{equation}
is admissible, where $\bar E = \Big[ {\begin{smallmatrix}
E&{}\\
{}&E
\end{smallmatrix}} \Big],~\bar x\left( t \right) = \Big[ {\begin{smallmatrix}
{x\left( t \right)}\\
{x\left( t \right) - \hat x\left( t \right)}
\end{smallmatrix}} \Big]$ and $\bar A = \Big[ {\begin{smallmatrix}
{A + BK}&{ - BK}\\
0&{A + LC}
\end{smallmatrix}} \Big].$

\vspace{3pt}
In preparation for designing $K$ and $L$, a useful property of $\mathbb{P}^{n\times n}_\alpha$ is introduced here, which can be regarded as a generalization of statement viii) Theorem 1 in \cite{Wei:2017FCAA}.
\begin{theorem}\label{Theorem1}
If $P \in \mathbb{P}_\alpha ^{n \times n}$, then for any $M \in {{\mathbb{R}}^{n \times r}}$ with ${\rm rank}\left(M\right)=r$, one has $M^{\rm T}PM\in\mathbb{P}_\alpha ^{r \times r}$.
\end{theorem}
\begin{proof}
From the given condition of $P$, there exist two matrices $X,Y\in\mathbb{R}^{n \times n}$ satisfying $P=\sin \big( {\frac{{\alpha \pi }}{2}} \big)X+\cos \big( {\frac{{\alpha \pi }}{2}} \big) Y$ and $\left[ {\begin{smallmatrix}
X&Y\\
{ - Y}&X
\end{smallmatrix}} \right] >0 $. By defining $X_1=M^{\rm T}XM$, $Y_1=M^{\rm T}YM$ and $P_1=M^{\rm T}PM$, it follows
\begin{equation}\label{Eq7}
{\textstyle P_1=\sin \big( {\frac{{\alpha \pi }}{2}} \big)X_1+\cos \big( {\frac{{\alpha \pi }}{2}} \big) Y_1,}
\end{equation}
where $X_1,Y_1,P_1\in{{\mathbb{R}}^{r \times r}}$.

Additionally, the big matrix can be expressed as
\begin{equation}\label{Eq8}
{\textstyle \Big[ {\begin{smallmatrix}
{{X_1}}&{{Y_1}}\\
{ - {Y_1}}&{{X_1}}
\end{smallmatrix}} \Big] = \Big[ {\begin{smallmatrix}
{{M^{\rm{T}}}}&{}\\
{}&{{M^{\rm{T}}}}
\end{smallmatrix}} \Big]\Big[ {\begin{smallmatrix}
X&Y\\
{ - Y}&X
\end{smallmatrix}} \Big]\Big[ {\begin{smallmatrix}
M&{}\\
{}&M
\end{smallmatrix}} \Big].}
\end{equation}

By using the singular value decomposition, one has
\begin{equation}\label{Eq9}
{\textstyle \Big[ {\begin{smallmatrix}
{{M^{\rm{T}}}}&{}\\
{}&{{M^{\rm{T}}}}
\end{smallmatrix}} \Big]=U^{\rm T}\Sigma\left[ {\begin{smallmatrix}
{{I_{2r}}}&{{0}}
\end{smallmatrix}} \right] V,}
\end{equation}
where $\Sigma$ is a $2r\times 2r$ positive diagonal matrix, $U$ and $V$ are orthogonal matrices with appropriate dimensions.

Because the similarity transformation does not change the eigenvalue of the target matrix, then
\begin{equation}\label{Eq10}
{\textstyle V \Big[ {\begin{smallmatrix}
X&Y\\
{ - Y}&X
\end{smallmatrix}} \Big]V^{\rm T}=V\Big[ {\begin{smallmatrix}
X&Y\\
{ - Y}&X
\end{smallmatrix}} \Big]V^{-1}>0.}
\end{equation}

Considering that the order principal minor determinant of a positive definite matrix is still positive definite, one has
\begin{equation}\label{Eq11}
{\textstyle \left[ {\begin{smallmatrix}
{{I_{2r}}}&{{0}}
\end{smallmatrix}} \right] V \Big[ {\begin{smallmatrix}
X&Y\\
{ - Y}&X
\end{smallmatrix}} \Big]V^{\rm T}\Big[ {\begin{smallmatrix}
I_{2r}\\
0
\end{smallmatrix}} \Big]>0.}
\end{equation}

From the fact that congruent transformation does not change the
sign of the matrix eigenvalue, the desired result follows immediately
\begin{equation}\label{Eq12}
{\textstyle \begin{array}{l}
\Big[ {\begin{smallmatrix}
{{M^{\rm{T}}}}&{}\\
{}&{{M^{\rm{T}}}}
\end{smallmatrix}} \Big]\Big[ {\begin{smallmatrix}
X&Y\\
{ - Y}&X
\end{smallmatrix}} \Big]\Big[ {\begin{smallmatrix}
M&{}\\
{}&M
\end{smallmatrix}} \Big]\\
 ={U^{\rm{T}}}\Sigma \left[ {\begin{smallmatrix}
{{I_{2r}}}&0
\end{smallmatrix}} \right]V\Big[ {\begin{smallmatrix}
X&Y\\
{ - Y}&X
\end{smallmatrix}} \Big]{V^{\rm{T}}}\Big[ {\begin{smallmatrix}
{{I_{2r}}}\\
0
\end{smallmatrix}} \Big]{\Sigma ^{\rm{T}}}U\\
 = {\big( {{\Sigma ^{\rm{T}}}U} \big)^{\rm{T}}}\left( {\left[ {\begin{smallmatrix}
{{I_{2r}}}&0
\end{smallmatrix}} \right]V\Big[ {\begin{smallmatrix}
X&Y\\
{ - Y}&X
\end{smallmatrix}} \Big]{V^{\rm{T}}}\Big[ {\begin{smallmatrix}
{{I_{2r}}}\\
0
\end{smallmatrix}} \Big]} \right){\Sigma ^{\rm{T}}}U\\
>0.
\end{array}}
\end{equation}
Hence, one can say that $P_1\in\mathbb{P}_\alpha ^{r \times r}$.
\end{proof}

With the help of Lemmas \ref{Lemma1}, \ref{Lemma2} and Theorem \ref{Theorem1}, the criterion on how to determine $L$ and $K$ can be presented.

\begin{theorem}\label{Theorem2}
Designing the controller (\ref{Eq5}) for system (\ref{Eq1}), the resulting closed-loop control system (\ref{Eq6}) is admissible if and only if there exist matrices $P_1,P_2 \in \mathbb{P}_\alpha ^{n \times n}$, $Q_1 \in {{\mathbb{R}}^{\left( {n - r} \right) \times n}}$, $Q_2 \in {{\mathbb{R}}^{n\times\left( {n - r} \right)}}$, $R_1\in\mathbb{R}^{m\times n}$ and $R_2\in\mathbb{R}^{n\times p}$, such that
\begin{equation}\label{Eq13}
{\rm{sym}}( {A{P_1}{E^{\rm{T}}} + A{E_1}{Q_1} + B{R_1}} ) < 0,
\end{equation}
\begin{equation}\label{Eq14}
{\rm{sym}}( {{E^{\rm{T}}}{P_2}A + {Q_2}{E_2}A + {R_2}C} ) < 0,
\end{equation}
where ${E_1} \in {{\mathbb{R}}^{n \times \left( {n - r} \right)}}$, ${E_2} \in {{\mathbb{R}}^{ \left( {n - r} \right)\times n}}$, ${\rm rank}\left(E_1\right)={\rm rank}\left(E_2\right)=n-r$, $E{E_1} = 0$, ${E_2}E= 0$ and the desired gains  are given by
\begin{equation}\label{Eq15}
K = {R_1}{( {{P_1}{E^{\rm{T}}} + {E_1}{Q_1}} )^{ - 1}},
\end{equation}
\begin{equation}\label{Eq16}
L = {( {{E^{\rm{T}}}{P_2} + {Q_2}{E_2}} )^{ - 1}}{R_2}.
\end{equation}
\end{theorem}
\begin{proof} The proof of this theorem can be presented from two aspects.

\textbf{Sufficiency.} From (\ref{Eq14}) and Lemma \ref{Lemma2}, it can be found that the system $E{{\mathscr D}^\alpha }x\left( t \right) = \left( {A + LC} \right)x\left( t \right)$ is admissible. By applying Lemma \ref{Lemma1}, there exist matrices ${\tilde P}_2 \in \mathbb{P}_\alpha ^{n \times n}$ and ${\tilde Q}_2 \in {{\mathbb{R}}^{\left( {n - r} \right) \times n}}$
satisfying
\begin{eqnarray}\label{Eq17}
\Theta={\rm{sym}}( {\left( {A + LC} \right){{\tilde P}_2}{E^{\rm{T}}} + \left( {A + LC} \right){E_1}{{\tilde Q}_2}} ) < 0.
\end{eqnarray}

Formula (\ref{Eq13}) can be equivalently expressed as
\begin{eqnarray}\label{Eq18}
\Xi={\rm{sym}}( {\left( {A + BK} \right){P_1}{E^{\rm{T}}} + \left( {A + BK} \right){E_1}{Q_1}} ) < 0.
\end{eqnarray}

Hence, there exists a scalar $\epsilon > 0$, such that the following LMI holds
\begin{equation}\label{Eq19}
{\textstyle \begin{array}{l}
{\rm{sym}}( {\bar AP{{\bar E}^{\rm{T}}} + \bar A{{\bar E}_0}Q} )\\
 = \left[ {\begin{array}{*{20}{c}}
\Xi &{ - \varepsilon BK( {{{\tilde P}_2}{E^{\rm{T}}} + {E_1}{{\tilde Q}_2}})}\\
{*}&{\varepsilon \Theta }
\end{array}} \right]\\
 < 0
\end{array}}
\end{equation}
 for $P = \Big[ {\begin{smallmatrix}
{{P_1}}&{}\\
{}&{\varepsilon {{\tilde P}_2}}
\end{smallmatrix}} \Big]$, $Q = \Big[ {\begin{smallmatrix}
{{Q_1}}&{}\\
{}&{\varepsilon {{\tilde Q}_2}}
\end{smallmatrix}} \Big]$ and ${E_0} = \Big[ {\begin{smallmatrix}
{{E_1}}&{}\\
{}&{{E_1}}
\end{smallmatrix}} \Big]$. Therefore, under the conditions in Theorem \ref{Theorem1}, the resulting closed-loop control system (\ref{Eq6}) is admissible.

\textbf{Necessity.} It can be known from Lemma \ref{Lemma1}, if the system (\ref{Eq6}) is admissible, then there exist matrices $P$, $Q$ and ${\bar E}_1$ with appropriate dimensions satisfying
\begin{equation}\label{Eq20}
{\rm{sym}}( {\bar AP{{\bar E}^{\rm{T}}} + \bar A{{\bar E}_1}Q} ) < 0.
\end{equation}

Without loss of generality, let us suppose those matrices satisfying $P = \Big[ {\begin{smallmatrix}
{{P_1}}&{{P_{12}}}\\
{{P_{21}}}&{{P_{22}}}
\end{smallmatrix}} \Big]$, $Q =\hspace{-2pt} \Big[ {\begin{smallmatrix}
{{Q_1}}&{{Q_{12}}}\\
{{Q_{21}}}&{{Q_{22}}}
\end{smallmatrix}} \Big]$, ${{\bar E}_1} = \Big[ {\begin{smallmatrix}
{{E_1}}&0\\
0&{{E_1}}
\end{smallmatrix}} \Big]$ and ${R_1} = K( {{P_1}{E^{\rm{T}}} + {E_1}{Q_1}} )$. By adopting Theorem \ref{Theorem1}, it can obtain that if $P \in \mathbb{P}_\alpha ^{2n \times 2n}$, then $\left[ {\begin{smallmatrix}
{{I_n}}&0
\end{smallmatrix}} \right]P\left[ {\begin{smallmatrix}
{{I_n}}\\
0
\end{smallmatrix}} \right]\in \mathbb{P}_\alpha ^{n \times n}$. Furthermore, one has
\begin{equation}\label{Eq21}
\begin{array}{l}
[ {\begin{array}{*{20}{c}}
{{I_n}}&0
\end{array}} ]{\rm{sym}}( {\bar AP{{\bar E}^{\rm{T}}} + \bar A{{\bar E}_1}Q} )\Big[ {\begin{array}{*{20}{c}}
{{I_n}}\\
0
\end{array}} \Big]\\
 = {\rm{sym}}( {\left( {A + BK} \right){P_1}{E^{\rm{T}}} + \left( {A + BK} \right){E_1}{Q_1}} )\\
 = {\rm{sym}}( {A{P_1}{E^{\rm{T}}} + A{E_1}{Q_1} + B{R_1}} )\\
 < 0,
\end{array}
\end{equation}
which implies the inequality (\ref{Eq13}).

In a similar way, by using Lemma \ref{Lemma2}, it follows that if the system (\ref{Eq6}) is admissible, then there exist matrices $P$, $Q$ and ${\bar E}_2$ with appropriate dimensions satisfying
\begin{equation}\label{Eq22}
{\rm{sym}}( {{{\bar E}^{\rm{T}}}P\bar A + Q{{\bar E}_2}\bar A} ) < 0.
\end{equation}

By assuming $P = \Big[ {\begin{smallmatrix}
{{P_{11}}}&{{P_{12}}}\\
{{P_{21}}}&{{P_2}}
\end{smallmatrix}} \Big]$, $Q = \Big[ {\begin{smallmatrix}
{{Q_{11}}}&{{Q_{12}}}\\
{{Q_{21}}}&{{Q_2}}
\end{smallmatrix}} \Big]$, ${{\bar E}_2} = \Big[ {\begin{smallmatrix}
{{E_2}}&{}\\
{}&{{E_2}}
\end{smallmatrix}} \Big]$ and ${R_2} = ( {{E^{\rm{T}}}{P_2} + {Q_2}{E_2}} )L$, it becomes
\begin{equation}\label{Eq23}
\begin{array}{l}
[ {\begin{array}{*{20}{c}}
0&{{I_n}}
\end{array}} ]{\rm{sym}}( {{{\bar E}^{\rm{T}}}P\bar A + Q{{\bar E}_2}\bar A} )\Big[ {\begin{array}{*{20}{c}}
0\\
{{I_n}}
\end{array}} \Big]\\
 = {\rm{sym}}( {{E^{\rm{T}}}{P_2}\left( {A + LC} \right) + {Q_2}{E_2}\left( {A + LC} \right)} )\\
 = {\rm{sym}}( {{E^{\rm{T}}}{P_2}A + {Q_2}{E_2}A + {R_2}C} )\\
 < 0,
\end{array}
\end{equation}
which leads to the inequality (\ref{Eq14}). All of the above discussions complete the proof of Theorem \ref{Theorem2}.
\end{proof}

\begin{remark}\label{Remark1}
With the introduction of $R_1$ and $R_2$, the coupling problem of multiple decision matrices has been solved well. Without adopting the Kronecker product, Theorem \ref{Theorem2} is more concise than that in \cite{Ji:2015ISA}. Compared with the result in \cite{Marir:2017IJCAS}, the proposed approach avoids the complex calculation successfully.
\end{remark}
\subsection{Output feedback control}
The objective of this subsection is to design an output feedback controller
\begin{equation}\label{Eq24}
u\left( t \right) = Fy\left( t \right),
\end{equation}
such that the resulting closed-loop control system
\begin{equation}\label{Eq25}
E{{\mathscr D}^\alpha }x\left( t \right) = \left( {A + BFC} \right)x\left( t \right)
\end{equation}
is admissible. First, one relevant lemma will be provided here, which will play a key role in dealing with the decoupling problem of matrix variables.
\begin{lemma}\label{Lemma3} {\rm \cite{Saadni:2006AJC}}
Let $\Phi $, $a$ and $b$ be given matrices with appropriate dimensions, thus
\begin{equation}\label{Eq26}
{\textstyle \left\{ {\begin{array}{*{20}{l}}
{\Phi  < 0}\\
{\Phi  + {\rm{sym}}( {a{b^{\rm{T}}}} ) < 0}
\end{array}} \right.}
\end{equation}
holds if and only if there exists an appropriate dimension matrix $G$ which satisfies
\begin{equation}\label{Eq27}
{\textstyle \left[ {\begin{array}{cc}
\Phi &{a + b{G^{\rm{T}}}}\\
\ast&{- G-G^{\rm{T}}}
\end{array}} \right] < 0.}
\end{equation}
\end{lemma}

\begin{theorem}\label{Theorem3}
Designing the controller (\ref{Eq24}) for system (\ref{Eq1}), the resulting closed-loop control system (\ref{Eq25}) is admissible, if and only if there exist matrices $P\in\mathbb{P}^{n\times n}_\alpha$, $Q \in {{\mathbb{R}}^{ n\times\left( {n - r} \right) }}$, $G\in \mathbb{R}^{m\times m}$ and $H\in\mathbb{R}^{m\times p}$, such that
\begin{equation}\label{Eq28}
{\textstyle \left[ {\begin{array}{cc}
\Phi &( {{E^{\rm{T}}}P + Q{E_2}} )B + {C^{\rm{T}}}{H^{\rm{T}}} - K_0^{\rm{T}}G^{\rm{T}} \\
{*}&{- G-G^{\rm{T}}}
\end{array}} \right] < 0}
\end{equation}
and the controller gain is given by
\begin{equation}\label{Eq29}
F = {G^{ - 1}}H,
\end{equation}
where $\Phi  = {\rm{sym}}( {{E^{\rm{T}}}P\left( {A + B{K_0}} \right) + Q{E_2}\left( {A + B{K_0}} \right)} )$, ${E_2} \in {{\mathbb{R}}^{\left( {n - r} \right)\times n }}$, ${\rm rank}\left(E_2\right)=n-r$, $E{E_2} = 0$ and $K_0$ is an intermediate matrix derived from ${K_0} = {Z}{( {{X}{E^{\rm{T}}} + {E_1}{Y}} )^{ - 1}}$. The matrices $X\in \mathbb{P}^{n\times n}_\alpha$, $Y\in \mathbb{R}^{\left(n-r\right)\times n}$ and $Z\in \mathbb{R}^{m\times n}$ satisfy the following LMI
\begin{equation}\label{Eq30}
{\rm{sym}}( {AX{E^{\rm{T}}} + A{E_1}Y + BZ} ) < 0,
\end{equation}
where ${E_1} \in {{\mathbb{R}}^{n \times \left( {n - r} \right)}}$, ${\rm rank}\left(E_1\right)=n-r$ and $E{E_1} = 0$.
\end{theorem}
\begin{proof}
\textbf{Sufficiency.}
Defining $a = ( {{E^{\rm{T}}}P + Q{E_2}} )B$ and $b = {C^{\rm{T}}}{F^{\rm{T}}} - K_0^{\rm{T}}$, then (\ref{Eq28}) can be rewritten as (\ref{Eq27}). By applying Lemma \ref{Lemma3}, the equivalent conditions can be expressed as
\begin{equation}\label{Eq31}
\Phi  = {\rm{sym}}\big( {{E^{\rm{T}}}{P}\left( {A + B{K_0}} \right) + {Q}{E_2\left( {A + B{K_0}} \right)} }\big) < 0,
\end{equation}
and
\begin{equation}\label{Eq32}
\begin{array}{rl}
\Phi  + {\rm{sym}}( {a{b^{\rm{T}}}} )=&\hspace{-6pt} {\rm{sym}}\big( {{E^{\rm{T}}}{P}\left( {A + B{K_0}} \right) + {Q}{E_2}\left( {A + B{K_0}} \right)} \big)\\
 &\hspace{-6pt}+ {\rm{sym}}\big( ( {{E^{\rm{T}}}P + Q{E_2}} )B ( {C^{\rm{T}}}{F^{\rm{T}}} - K_0^{\rm{T}})^{\rm{T}}\big)\\
=&\hspace{-6pt} {\rm{sym}}\big( {{E^{\rm{T}}}{P}( {A + B{K_0}} ) + {Q}{E_2}( {A + B{K_0}} )} \big)\\
 &\hspace{-6pt}+ {\rm{sym}}\big( {{E^{\rm{T}}}P( {BFC - B{K_0}} )} + { Q{E_2}( {BFC - B{K_0}} )} \big)\\
=&\hspace{-6pt}{\rm{sym}}\big( {{E^{\rm{T}}}{P}( {A + BFC} ) + {Q}{E_2}( {A + BFC} )} \big)\\
<&\hspace{-6pt} 0.
\end{array}
\end{equation}
Luckily, with the help of Lemma \ref{Lemma2}, the previous formula (\ref{Eq32}) clearly indicates that the closed-loop control system in (\ref{Eq25}) is admissible.

\textbf{Necessity.} Using Lemma \ref{Lemma2}, if the system (\ref{Eq25}) is admissible, then there exist matrices $P\in\mathbb{P}^{n\times n}_\alpha$ and $Q \in {{\mathbb{R}}^{ n\times\left( {n - r} \right) }}$ satisfying
\begin{equation}\label{Eq33}
{\rm{sym}}( {{E^{\rm{T}}}{P}\left( {A + BFC} \right) + {Q}{E_2}\left( {A + BFC} \right)} ) < 0.
\end{equation}

Based on the basic property of open set, there must exist small positive number $\epsilon$, such that
\begin{equation}\label{Eq34}
{\rm{sym}}( {{E^{\rm{T}}}{P}\left( {A + BFC + \epsilon I_n} \right) + {Q}{E_2}\left( {A + BFC + \epsilon I_n} \right)} ) < 0.
\end{equation}
When $K_\epsilon\in \mathbb{R}^{m\times n}$ and $2\epsilon I_n-{\rm{sym}}(BK_\epsilon)>0$, it follows
\begin{equation}\label{Eq35}
{\rm{sym}}( {{E^{\rm{T}}}{P}\left( {A + BFC + BK_\epsilon} \right) + {Q}{E_2}\left( {A + BFC + BK_\epsilon} \right)} ) < 0.
\end{equation}
Defining $K_0=FC + K_\epsilon$, one has
\begin{equation}\label{Eq36}
{\rm{sym}}( {{E^{\rm{T}}}{P}\left( {A + BK_0} \right) + {Q}{E_2}\left( {A + BK_0} \right)} ) < 0.
\end{equation}
By using the equivalence of formula (\ref{Eq26}) and formula (\ref{Eq27}), the desirable result (\ref{Eq28}) can be deduced immediately from formulas (\ref{Eq33}) and (\ref{Eq36}).

Note that, formula (\ref{Eq36}) means the following system
\begin{equation}\label{Eq37}
E{{\mathscr D}^\alpha }x\left( t \right) = \left( {A + B{K_0}} \right)x\left( t \right)
\end{equation}
is admissible. By adopting Lemma \ref{Lemma1}, one can conclude that there must exist matrices $X\in \mathbb{P}^{n\times n}_\alpha$ and $Y\in \mathbb{R}^{\left(n-r\right)\times n}$ satisfying
\begin{equation}\label{Eq38}
{\rm{sym}}( {(A+BK_0)X{E^{\rm{T}}} + (A+BK_0){E_1}Y} ) < 0.
\end{equation}
Setting $Z=K_0(XE^{\rm T}+E_1Y)$, (\ref{Eq30}) follows from (\ref{Eq38}) immediately.

All of these complete the proof of Theorem \ref{Theorem3}.
\end{proof}

\begin{remark}\label{Remark2}
Although \cite{Marir:2017JFI} has studied the output feedback control problem for singular FOS with $1<\alpha<2$, the main result in Theorem \ref{Theorem3} is questionable. With the introduction of relaxation matrix $G$, the coupling problem of decision matrices has been solved. Notably, a sufficient and necessary condition is systematically proposed for constructing output feedback controller. The proposed result in this study can not only solve the static output feedback problem but also solve the dynamic output feedback problem. The interested readers can refer to the reference \cite{Wei:2017ISA}.
\end{remark}

\subsection{Implementing strategy}

From \cite{Duan:2010Book}, it can be found that if the pair $\{E,A\}$ in (\ref{Eq1}) is regular, then there exist two invertible matrices satisfying
\begin{equation}\label{Eq39}
E = M\left[ {\begin{array}{*{20}{c}}
{{I_r}}&{}\\
{}&0
\end{array}} \right]N,A = M\left[ {\begin{array}{*{20}{c}}
{{{ A}_1}}&{{{ A}_2}}\\
{{{ A}_3}}&{{{ A}_4}}
\end{array}} \right]N,
\end{equation}
where ${\rm rank}({A}_4)=n-r$.

Defining $\tilde x\left( t \right) = \big[ {\begin{smallmatrix}
{{x_a}\left( t \right)}\\
{{x_b}\left( t \right)}
\end{smallmatrix}} \big] \triangleq Nx\left( t \right),\tilde B = \big[ {\begin{smallmatrix}
{{B_1}}\\
{{B_2}}
\end{smallmatrix}} \big] \triangleq {M^{ - 1}}B$, then the system in (\ref{Eq1}) can be rewritten as
\begin{equation}\label{Eq40}
\left[ {\begin{array}{*{20}{c}}
{{I_r}}&{}\\
{}&0
\end{array}} \right]\left[ {\begin{array}{*{20}{c}}
{{{\mathscr D}^\alpha }{x_a}\left( t \right)}\\
{{{\mathscr D}^\alpha }{x_b}\left( t \right)}
\end{array}} \right] = \left[ {\begin{array}{*{20}{c}}
{{A_1}}&{{A_2}}\\
{{A_3}}&{{A_4}}
\end{array}} \right]\left[ {\begin{array}{*{20}{c}}
{{x_a}\left( t \right)}\\
{{x_b}\left( t \right)}
\end{array}} \right] + \left[ {\begin{array}{*{20}{c}}
{{B_1}}\\
{{B_2}}
\end{array}} \right]u\left( t \right).
\end{equation}

Let us record ${A_a} = {A_1} - {A_2}A_4^{ - 1}{A_3},{B_a} = {B_1} - {A_2}A_4^{ - 1}{B_2},{A_b} =  - A_4^{ - 1}{A_3}$ and ${B_b}= - A_4^{ - 1}{B_2}$. Then the algebraic equation and the differential equation in (\ref{Eq40}) can be equivalently expressed as
\begin{equation}\label{Eq41}
{{\mathscr D}^\alpha }{x_a}\left( t \right) = {A_a}{x_a}\left( t \right) + {B_a}u\left( t \right),
\end{equation}
\begin{equation}\label{Eq42}
{x_b}\left( t \right) = {A_b}{x_a}\left( t \right) + {B_b}u\left( t \right).
\end{equation}
Note that $x_a\left( t \right)$ is governed by a regular fractional order differential equation and $x_b\left( t \right)$ can be calculated from $x_a\left( t \right)$ and $u\left( t \right)$ directly. Theorem \ref{Theorem2} and Theorem \ref{Theorem3} are both applicable to the case of $\alpha\in(0,1)$. In this case, the investigated singular fractional order control system can be implemented successfully with the methods in \cite{Wei:2016ISA}. Actually, when $\alpha=1$, $\mathbb{P}^{n\times n}_{\alpha}$ reduces to the classical positive definite matrix set and the obtained results still hold. The corresponding implementation problem can be regarded as a special case of the previous one.

Because the fundamental Lemma \ref{Lemma1} and Lemma \ref{Lemma2} require $\alpha\in(0,1]$, some measures should be taken specially for $\alpha\in(1,2)$. For any $k\in\mathbb{N}_+$, an auxiliary system can be built as
\begin{equation}\label{Eq43}
\left\{ \begin{array}{rl}
E{{\mathscr D}^{{\alpha  \mathord{\left/
 {\vphantom {\alpha  k}} \right.
 \kern-\nulldelimiterspace} k}}}z_1\left( t \right) = &\hspace{-6pt}{z_2}\left( t \right),\\
{{\mathscr D}^{{\alpha  \mathord{\left/
 {\vphantom {\alpha  k}} \right.
 \kern-\nulldelimiterspace} k}}}{z_i}\left( t \right) =&\hspace{-6pt} {z_{i+1}}\left( t \right),i=2,3,\cdots,k-1,\\
{{\mathscr D}^{{\alpha  \mathord{\left/
 {\vphantom {\alpha  k}} \right.
 \kern-\nulldelimiterspace} k}}}{z_k}\left( t \right) =&\hspace{-6pt} Az_1\left( t \right) + Bu\left( t \right),\\
y\left( t \right) =&\hspace{-6pt} Cz_1\left( t \right).
\end{array} \right.
\end{equation}

Due to the fact that the transfer function of the system (\ref{Eq1}) is equivalent to that of the system (\ref{Eq43}), namely,
\begin{equation}\label{Eq44}
G\left( s \right) = C{\left( {{s^\alpha }E - A} \right)^{ - 1}}B,
\end{equation}
and the initial conditions of linear singular FOSs do not affect their admissibility. Note that $\alpha/k\in(0,1)$ for $k\ge 2$, and therefore the admissibility of system (\ref{Eq1}) can be determined from the system (\ref{Eq43}). In other words, the procedure for dealing with the system with $\alpha\in(1,2)$ can be summarized as follows. i) With the help of the model transform in (\ref{Eq43}), then an auxiliary system can be constructed. ii) By applying Theorem \ref{Theorem2} or Theorem \ref{Theorem3}, the needed controller can be designed. iii) By using the model transform in (\ref{Eq40}), an easy to implement model can be obtained. Although the presented criteria are suitable for both $0<\alpha<1$ and $1<\alpha<2$, the main results are not independent of $\alpha$ since the critical set $\mathbb{P}_\alpha ^{n \times n}$ is relevant to $\alpha$.

Besides the Caputo derivative, the following Riemann--Liouville derivative is also widely used
\begin{equation}\label{Eq45}
\begin{array}{l}
{\mathscr D}^\alpha f\left( t \right) \triangleq \frac{{{{\rm{d}}^n}}}{{{\rm{d}}{t^n}}}\frac{1}{{\Gamma \left( {n - \alpha } \right)}}\int_{0}^t {{{\left( {t - \tau } \right)}^{n - \alpha  - 1}}f\left( \tau  \right){\rm{d}}\tau },
\end{array}
\end{equation}
where $n-1<\alpha<n$ and $n\in\mathbb{N}_+$. It has shown that an FOS with Caputo derivative and Riemann--Liouville derivative are equivalent in the form of infinite dimensional state space model except the initial conditions \cite{Wei:2016ISA}. Additionally, the initial conditions do not affect the admissibility. As a result, the elaborated approaches still can be extended to the Riemann--Liouville case.

\begin{remark}\label{Remark3}
Before ending this main results section, it is worth mentioning the main contributions of this work. i) An essential property on fractional order positive definite matrix is derived, which  plays an important role in the sequel. ii) A sufficient and necessary condition for designing observer based controller is proposed, which needs fewer real decision matrices. iii) A sufficient and necessary condition for designing output feedback controller is developed, which is the first time to give such a complete condition. iv) After clearly providing the implementing strategy, the results are extended to the case of $1<\alpha<2$.
\end{remark}

\section{Numerical Example}\label{Section4}
In this section, we provide two numerical examples to illustrate the applicability of the proposed method. To enhance the persuasion, consider the benchmark system from \cite{Marir:2017JFI} as follows
\begin{equation}\label{Eq46}
\left\{ \begin{array}{rl}
\left[ {\begin{array}{*{20}{l}}
1&1&1\\
0&1&1\\
0&0&0
\end{array}} \right]{{\mathscr D}^\alpha }x\left( t \right) =&\hspace{-6pt} \left[ {\begin{array}{*{20}{l}}
1&1&{ - 1}\\
2&{ - 2}&{ - 1}\\
4&1&{ - 4}
\end{array}} \right]x\left( t \right) + \left[ {\begin{array}{*{20}{l}}
1\\
1\\
1
\end{array}} \right]u\left( t \right),\\
y\left( t \right) =&\hspace{-6pt} [ {\begin{array}{*{20}{l}}
1&0&1
\end{array}} ]x\left( t \right),
\end{array} \right.
\end{equation}
which has been shown as regular, impulse free and unstable with zero input. To guarantee the system state change smoothly, a zero input should be provided at $t=0$, thus $\hat x(0)=0$ and $y(0)=0$ are initialized. To make the equation (\ref{Eq46}) hold without input, $[ {\begin{array}{*{20}{c}}
{4}&{1}&{-4}
\end{array}} ]x(t)=0$ is expected. Without loss of generality, let us configure the initial pseudo state as $x(0)= [ {\begin{array}{*{20}{c}}
{-0.25}&{2}&{0.25}
\end{array}} ]^{\rm T}$.

\begin{example}
Consider the singular FOS {\rm (\ref{Eq46})} with $\alpha=0.6$. Firstly, $M = \left[ {\begin{smallmatrix}
{{\rm{1.6834}}}&{ - {\rm{0.4075}}}&0\\
{{\rm{1.3144}}}&{{\rm{0.5219}}}&0\\
0&0&1
\end{smallmatrix}} \right]$, $N = \left[ {\begin{smallmatrix}
{{\rm{0.369}}}&{{\rm{0.6572}}}&{{\rm{0.6572}}}\\
{ - {\rm{0.9294}}}&{{\rm{0.261}}}&{{\rm{0.261}}}\\
0&{ - {\rm{0.7071}}}&{{\rm{0.7071}}}
\end{smallmatrix}} \right]$, $E_1= [ {\begin{array}{*{20}{c}}
{0}&{1}&{-1}
\end{array}} ]^{\rm T}$ and $E_2= [ {\begin{array}{*{20}{c}}
{0}&{0}&{1}
\end{array}} ]$ are chosen to satisfy $E=M\left[ {\begin{smallmatrix}
{{I_2}}&{}\\
{}&0
\end{smallmatrix}} \right]N$, $E{E_1} = 0$ and ${E_2}E = 0$. The initial condition of the observer is set as $\hat x(0)=[ {\begin{array}{*{20}{c}}
{0}&{0}&{0}
\end{array}} ]^{\rm T}$. Then, by applying Theorem \ref{Theorem2}, the needed controller gain and observer gain
can be obtained as
\[\begin{array}{l}
K = [ {\begin{array}{*{20}{c}}
{-3.1656}&{-0.4720}&{2.4146}
\end{array}} ],\\
L = {[ {\begin{array}{*{20}{c}}
{ -0.1821}&{0.0996}&{0.7768}
\end{array}} ]^{\rm{T}}},
\end{array}\]
respectively. To confirm the effectiveness of the obtained results, the curves of the system state $x(t)$ and the control input $u(t)$ are given in Fig. \ref{Fig1} and Fig. \ref{Fig2}, respectively.
\begin{figure}[!htbp]
\centering
\includegraphics[width=0.8\columnwidth]{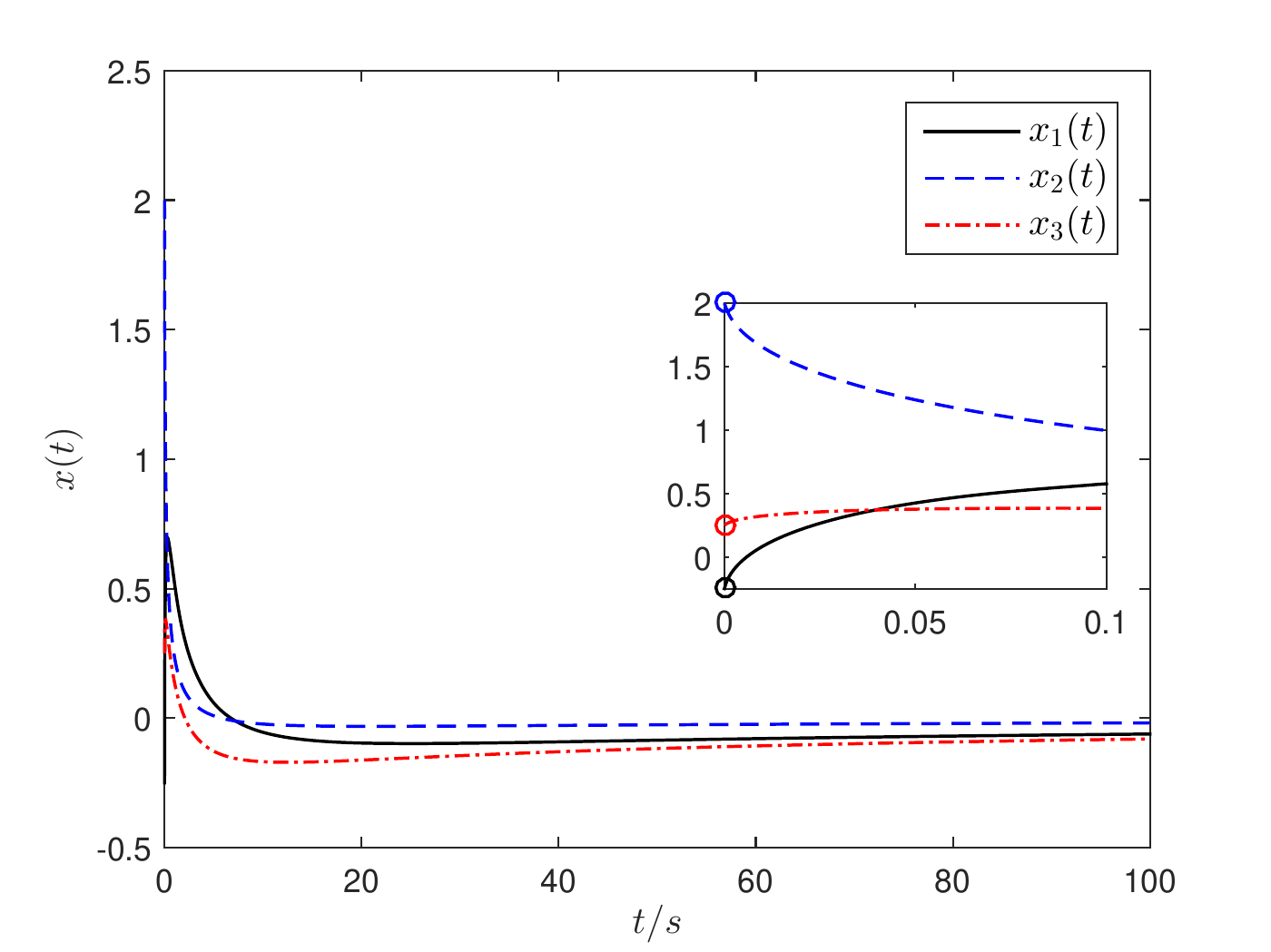}\vspace{-10pt}
\caption{$x(t)$ for system (\ref{Eq46}) under observer based control ($\alpha=0.6$).}\label{Fig1}
\end{figure}
\begin{figure}[!htbp]
\centering
\includegraphics[width=0.8\columnwidth]{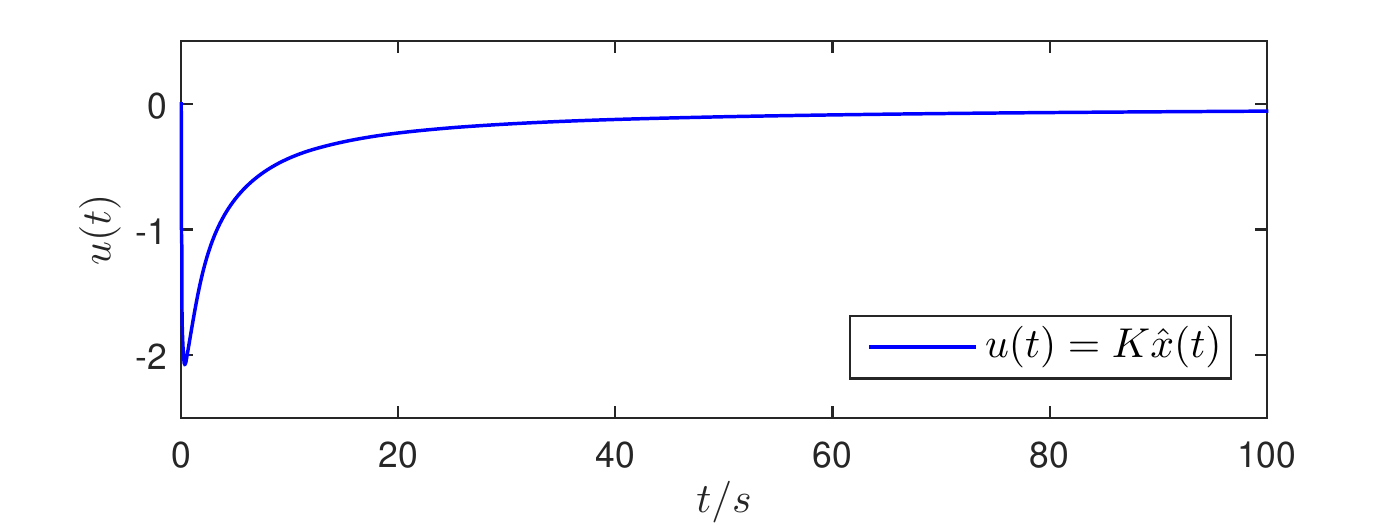}\vspace{-10pt}
\caption{$u(t)$ for system (\ref{Eq46}) under observer based control ($\alpha=0.6$).}\label{Fig2}
\end{figure}

The admissibility of the closed-loop control system in (\ref{Eq6}) can be seen clearly \textcolor[rgb]{0,0,1}{from} these two figures. Additionally, the convergence speed is slowly as claimed in \cite{Monje:2010Book} and is approximately equal to $t^{-\alpha}$. From the enlarged view, it can be observed that the initial conditions coincide with the setting value. 
To show the observation performance clearly, the observation error is defined as $e(t)\triangleq x(t)-\hat x(t)$ and its curve is displayed in Fig. \ref{Fig3}, which clearly demonstrates the validity of the designed observer. It can be found that the initial observation error $e(0)$ equals to $x(0)$, since zero initial conditions are exerted on the developed observer.
\begin{figure}[!htbp]
\centering
\includegraphics[width=0.8\columnwidth]{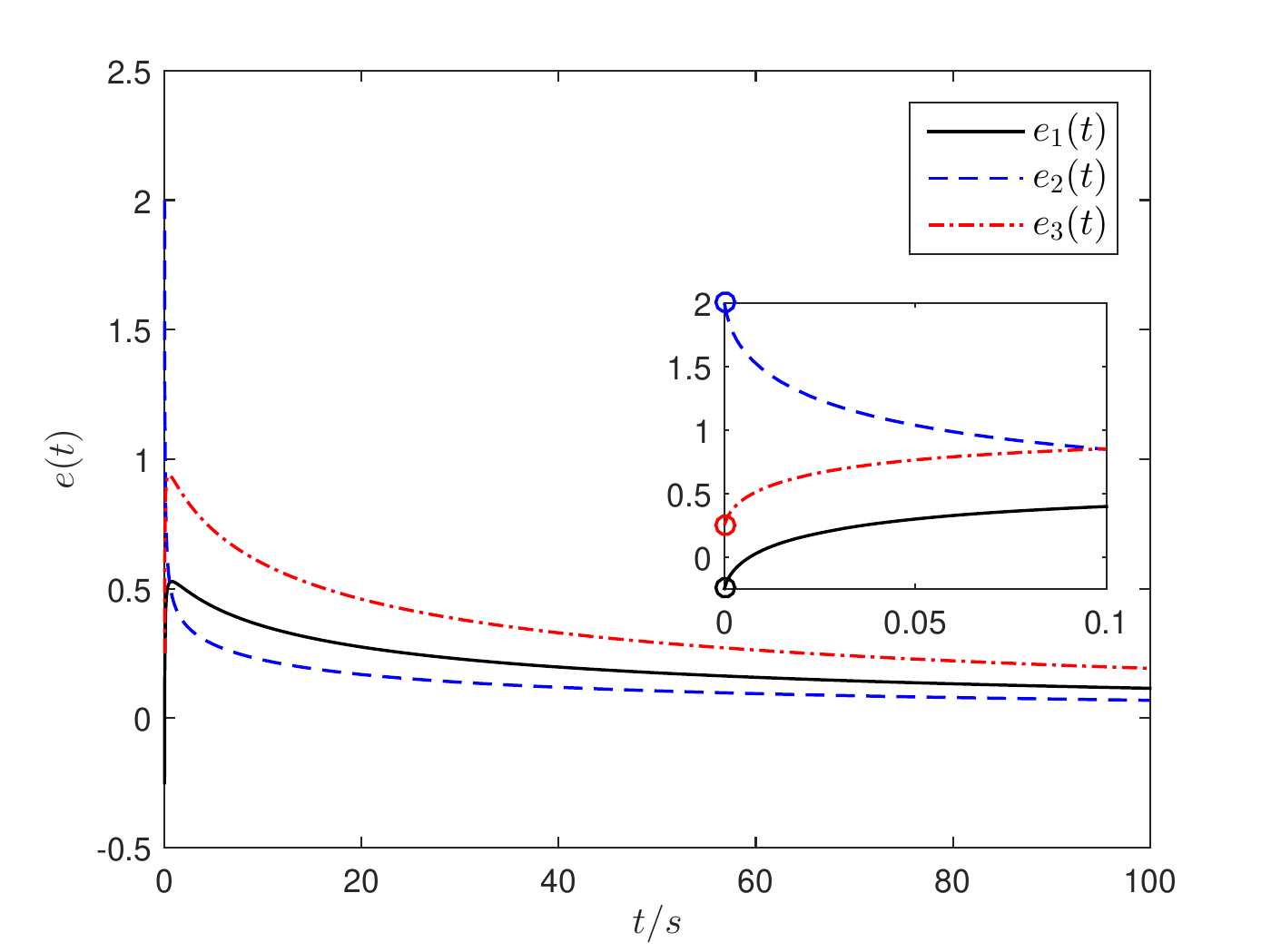}\vspace{-10pt}
\caption{The observation error of $x(t)$ for system (\ref{Eq46}) ($\alpha=0.6$).}\label{Fig3}
\end{figure}

From Theorem \ref{Theorem3}, the feasible controller gains can be obtained as follows
\[\begin{array}{l}
K_0 = [ {\begin{array}{*{20}{c}}
{-3.1656}&{-0.4720}&{2.4146}
\end{array}} ],\\
F = -3.6723.
\end{array}\]
The time response of the closed-loop control system (\ref{Eq25}) with output feedback control law (\ref{Eq24}) is illustrated in Fig. \ref{Fig4} and the
corresponding control input is \textcolor[rgb]{0,0,1}{shown} in Fig. \ref{Fig5}, which \textcolor[rgb]{0,0,1}{reveals} that the resulting system is admissible and its state $x(t)$ converge to zero as $t\to +\infty$. Note that the initial value of $x(t)$ indeed appears as predefined.
\begin{figure}[!htbp]
\centering
\includegraphics[width=0.8\columnwidth]{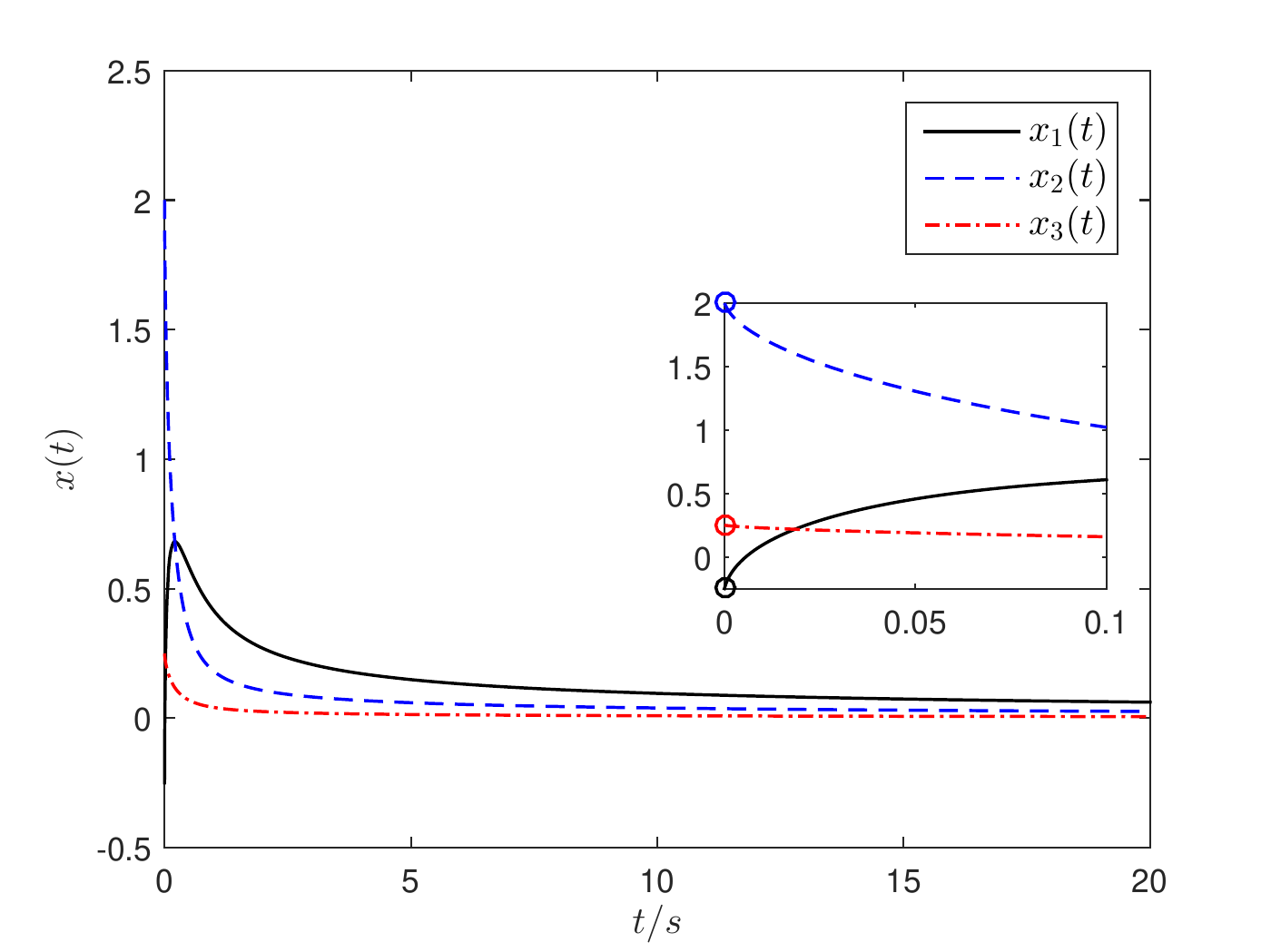}\vspace{-10pt}
\caption{$x(t)$ for system (\ref{Eq46}) under output feedback control ($\alpha=0.6$).}\label{Fig4}
\end{figure}
\begin{figure}[!htbp]
\centering
\includegraphics[width=0.8\columnwidth]{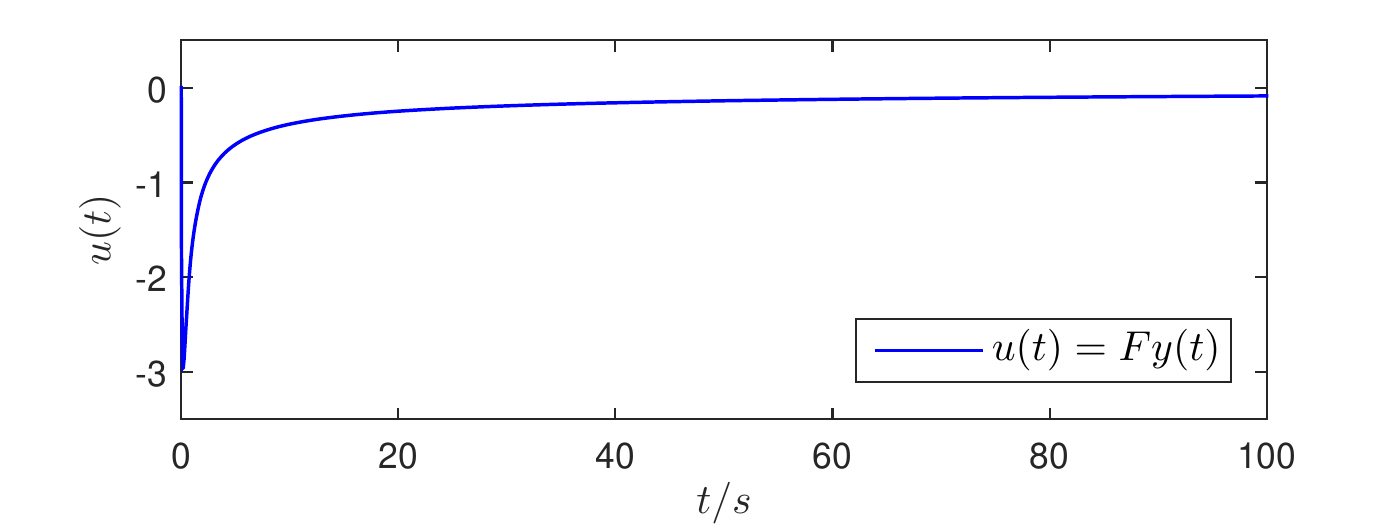}\vspace{-10pt}
\caption{$u(t)$ for system (\ref{Eq46}) under output feedback control ($\alpha=0.6$).}\label{Fig5}
\end{figure}
\end{example}

\begin{example}
Consider the singular FOS (\ref{Eq46}) with the related parameters $\alpha=1.2$. $k=2$ is selected for the model transform scheme in (\ref{Eq43}). The corresponding matrices $M,N,E_1,E_2$ can be designed in accordance with the principles of Example 1. To ensure the admissability of the closed-loop system, the feasible parameters are designed as
\[\begin{array}{l}
K = \left[-0.8663~-0.2339~-0.2990~-1.0001~-0.7116~0.2144~\right],\\
L = \left[-1.7022~0.1766~-0.0905~-4.059~-0.0028~-6.4078~\right]^{\rm{T}},\\
K_0 = \left[-0.8663~-0.2339~-0.2990~-1.0001~-0.7116~0.2144~\right],\\
F = -0.9515.
\end{array}\]
It should point out that with the help of (\ref{Eq43}), the desired controller is designed while the simulation is performed on the original system (\ref{Eq46}) with $\alpha=1.2$. Because $\alpha>1$, an additional initial condition $\dot x(0)=[ {\begin{array}{*{20}{c}}
{0}&{0}&{0}
\end{array}} ]^{\rm T}$ is given except $x(0)$. The time response of the closed-loop systems (\ref{Eq6}) with the designed controller is illustrated in Fig. \ref{Fig6}, the corresponding control input is given in Fig. \ref{Fig7}, and the state observation performance is provided in Fig. \ref{Fig8}. All of these show that the observer and the corresponding controller perform well. For the output feedback control case, the time response of the closed-loop system (\ref{Eq25}) is plotted in Fig. \ref{Fig9}, and the corresponding control input is given in Fig. \ref{Fig10}, which show that it is asymptotically stable and its states converge to zero gradually. Additionally, compared with the $0<\alpha<1$ case, $x(t)$ in the case of $1<\alpha<2$ converge more rapidly with comparable control, while the overshoot occur as expected. All the conducted simulations study shows the effectiveness of the proposed control scheme.

\begin{figure}[!htbp]
\centering
\includegraphics[width=0.8\columnwidth]{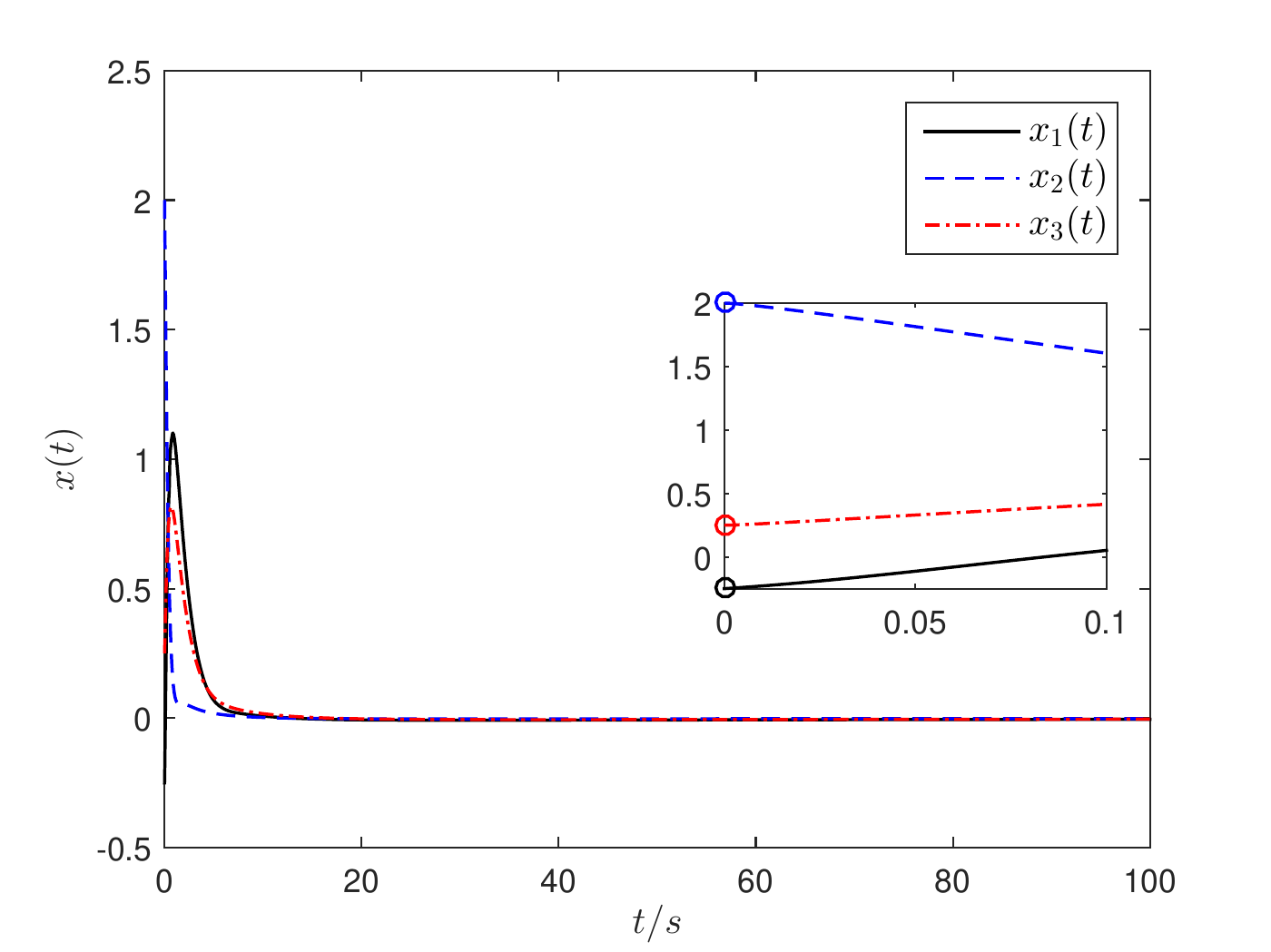}\vspace{-10pt}
\caption{$x(t)$ for system (\ref{Eq46}) under observer based control ($\alpha=1.2$).}\label{Fig6}
\end{figure}
\begin{figure}[!htbp]
\centering
\includegraphics[width=0.8\columnwidth]{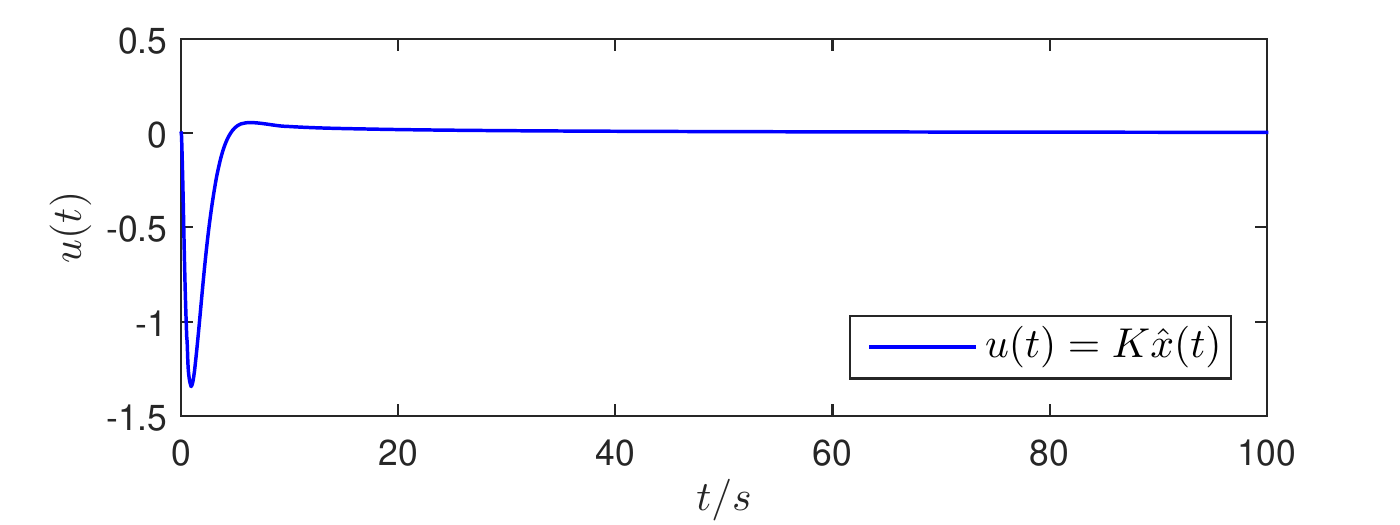}\vspace{-10pt}
\caption{$u(t)$ for system (\ref{Eq46}) under observer based control ($\alpha=1.2$).}\label{Fig7}
\end{figure}
\begin{figure}[!htbp]
\centering
\includegraphics[width=0.8\columnwidth]{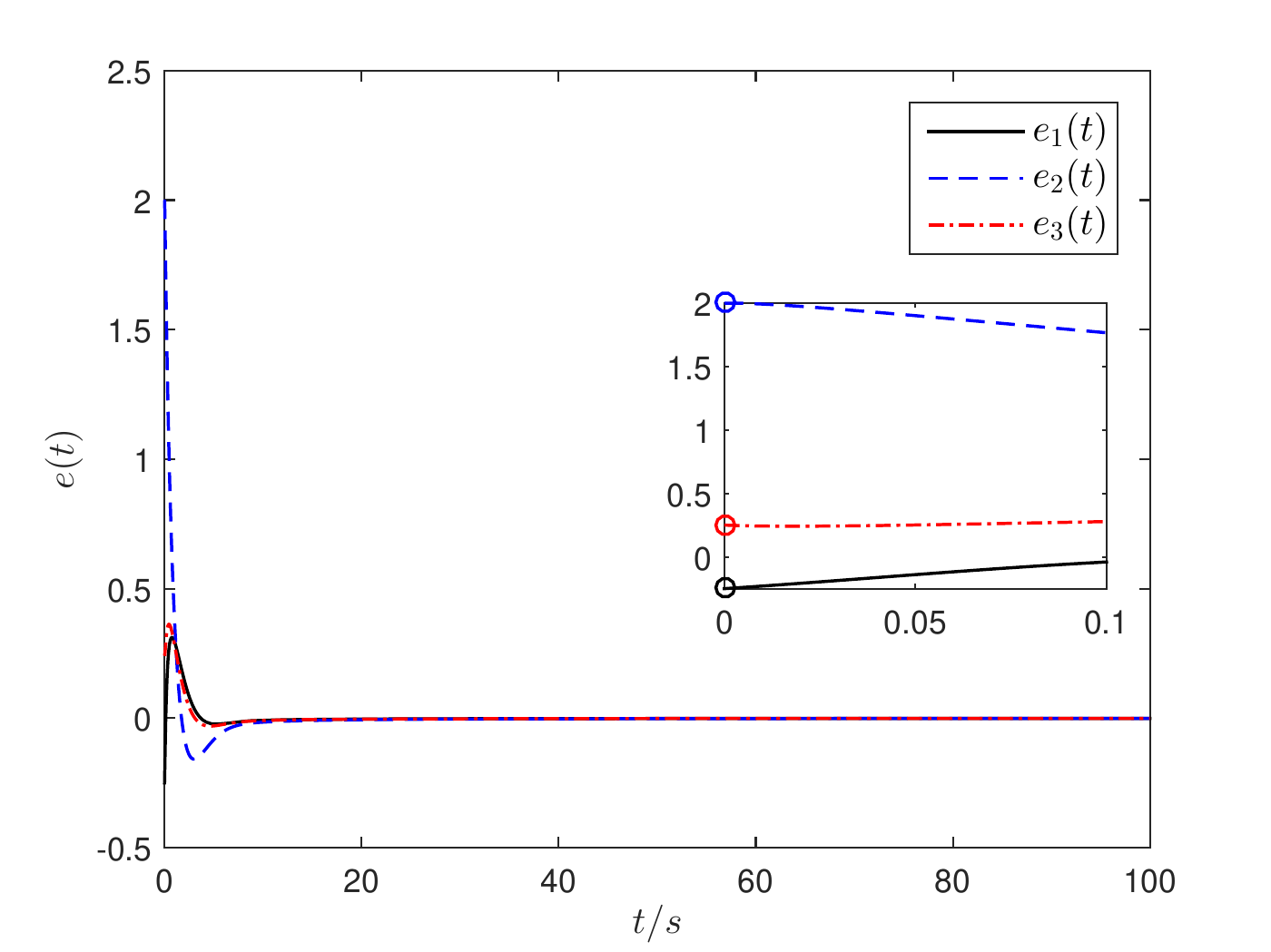}\vspace{-10pt}
\caption{The observation error of $x(t)$ for system (\ref{Eq46}) ($\alpha=1.2$).}\label{Fig8}
\end{figure}
\begin{figure}[!htbp]
\centering
\includegraphics[width=0.8\columnwidth]{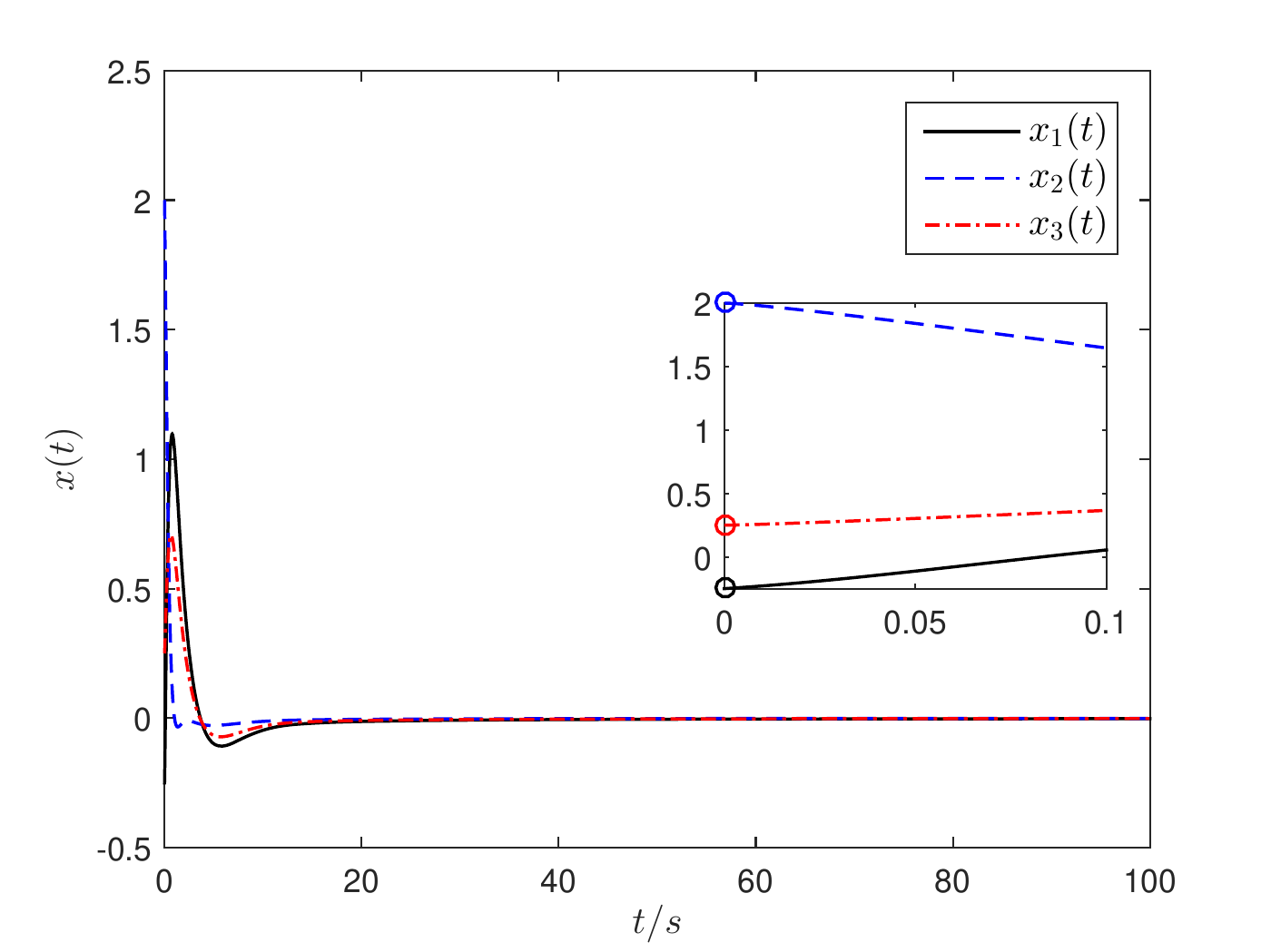}\vspace{-10pt}
\caption{$x(t)$ for system (\ref{Eq46}) under output feedback control ($\alpha=1.2$).}\label{Fig9}
\end{figure}
\begin{figure}[!htbp]
\centering
\includegraphics[width=0.8\columnwidth]{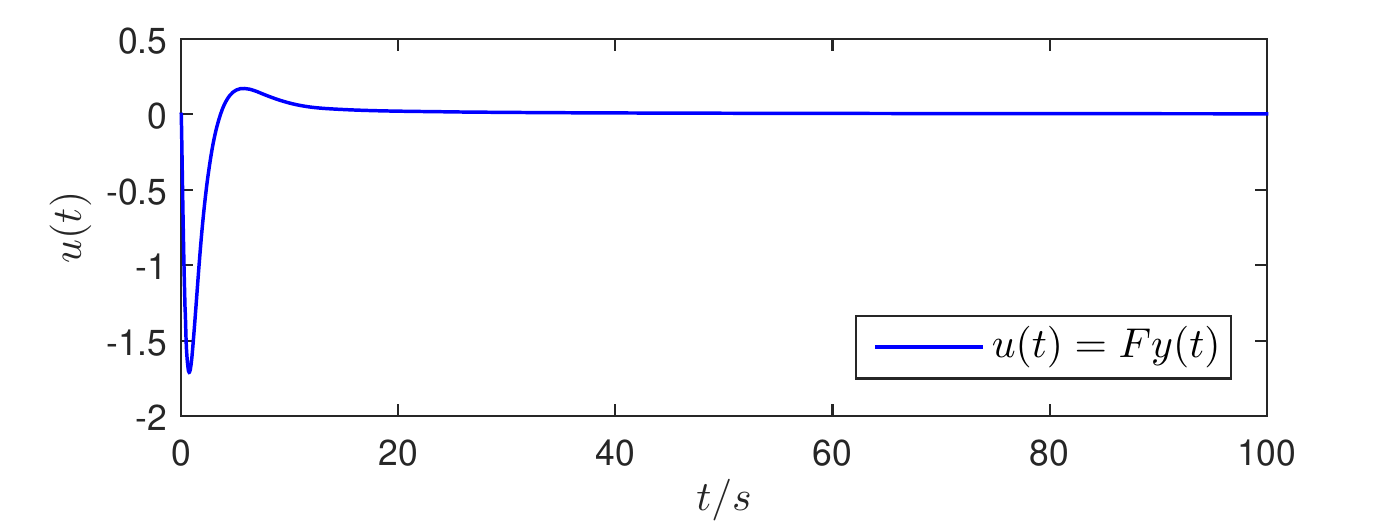}\vspace{-10pt}
\caption{$u(t)$ for system (\ref{Eq46}) under output feedback control ($\alpha=1.2$).}\label{Fig10}
\end{figure}
\end{example}

\section{Conclusions}\label{Section5}
In this paper, the stabilization issue of singular fractional order systems with partial state feedback has been studied. By the aid of the admissibility criteria we proposed before, two sufficient and necessary criteria are derived for the observer based control case and the output feedback control case, respectively. The results are expressed in terms of LMI, which are convenient to use. Furthermore, the applicability and implementation are discussed. Numerical simulation has shown the effectiveness of the theoretic results obtained. It is believed that this study provides a promising way to solve such problems of singular fractional order systems.

%
\section*{Acknowledgements}
The work described in this paper was supported by the National Natural Science Foundation of China (61601431, 61573332), the Anhui Provincial Natural Science Foundation (1708085QF141), the fund of China Scholarship Council (201806345002), the Fundamental Research Funds for the Central Universities (WK2100100028) and the General Financial Grant from the China Postdoctoral Science Foundation (2016M602032).

%

\phantomsection
\addcontentsline{toc}{section}{References}
\section*{References}
\bibliographystyle{model1-num-names}
\bibliography{database}

\end{document}